\UseRawInputEncoding
 \documentclass[final,1p,times]{elsarticle}

\usepackage{amssymb}
 \usepackage{amsthm}
\usepackage{amsmath,amssymb,amsopn,amsfonts,mathrsfs,amsbsy,amscd}

\makeatletter

\@addtoreset{equation}{section}
\makeatother

\journal{}

\newcommand{\Rr}{\mathrm{R}}
\newcommand{\om}{\omega}
\newcommand{\esp}{\quad\mbox{and}\quad}

\newcommand{\End}{\mathrm{End}}
\newcommand{\ass}{\mathrm{ass}}

\newcommand{\G}{{\mathfrak{g}}}

\newcommand{\h}{{\mathfrak{h}}}
\newcommand{\ad}{{\mathrm{ad}}}

\newcommand{\tr}{{\mathrm{tr}}}

\newcommand{\Ll}{{\mathrm{L}}}

\newcommand{\Om}{\Omega}
\newcommand{\na}{\nabla}
\newcommand{\wi}{\widetilde}

\font\bb=msbm10

\def\R{\hbox{\bb R}}

\newtheorem{Def}{Definition}[section]
\newtheorem{theo}{Theorem}[section]
\newtheorem{pr}{Proposition}[section]
\newtheorem{Le}{Lemma}[section]

\newtheorem{exem}{Example}[section]
\newtheorem{rem}{Remark}[section]
\begin{document}

\begin{frontmatter}




\title{Flat symplectic Lie algebras}

 \author[Boucetta]{Mohamed Boucetta}
\address[Boucetta]{Universit\'e Cadi-Ayyad,
	Facult\'e des sciences et techniques,
	B.P. 549, Marrakech, Maroc.\\m.boucetta@uca.ac.ma
}
\author[Hamza]{Hamza El Ouali}
\address[Hamza]{Universit\'e Cadi-Ayyad,
	Facult\'e des sciences et techniques,
	B.P. 549, Marrakech, Maroc.\\eloualihamza11@gmail.com
}
\author[Hicham Lebzioui]{Hicham Lebzioui\corref{mycorrespondingauthor}}
\cortext[mycorrespondingauthor]{Corresponding author}
\address[Hicham Lebzioui]{Universit\'e Sultan Moulay Slimane,
	\'{E}cole Sup\'{e}rieure de Technologie-Kh\'{e}nifra, B.P : 170, Kh\'{e}nifra, Maroc.\\h.lebzioui@usms.ma
}

\begin{abstract}
Let $(G,\Omega)$ be a symplectic Lie group, i.e, a Lie group endowed with a left invariant symplectic form.  If $\G$ is the Lie algebra of $G$ then we call $(\G,\omega=\Om(e))$ a symplectic  Lie algebra. The product $\bullet$ on $\G$ defined by $3\omega\left(x\bullet y,z\right)=\omega\left([x,y],z\right)+\omega\left([x,z],y\right)$ extends to a left invariant connection $\na$ on $G$ which  is torsion free and  symplectic ($\na\Om=0)$. When $\na$ has vanishing curvature, we call $(G,\Omega)$ a flat symplectic Lie group and $(\G,\om)$ a flat symplectic Lie algebra.  In this paper, we study flat symplectic Lie groups. We start by showing that the derived ideal of a flat symplectic Lie algebra is degenerate with respect to $\om$. We show that a flat symplectic Lie group must be nilpotent with degenerate center. This implies that the connection $\na$ of a flat symplectic Lie group is always complete. We prove that the double extension process can be applied to characterize all flat symplectic Lie algebras. More precisely, we show that every flat symplectic Lie algebra is obtained by a sequence of double extension of flat symplectic Lie algebras starting from $\{0\}$. As examples in low dimensions, we classify  all flat symplectic Lie algebras of dimension $\leq6$.  
\end{abstract}


\end{frontmatter}






\section{Introduction}
Let $(M,\Om)$ be a symplectic manifold, i.e., a manifold endowed with a nondegenerate closed 2-form. A {\it symplectic} connection is a connection $\na$ such that $\na\Om=0$. Contrasting with pseudo-Riemannian manifolds where there exits a unique torsion free connection which preserves the metric, namely, the Levi-Civita connection, on a symplectic manifold, there are many torsion free symplectic connections (see \cite{cahen}). However, in a symplectic Lie group $(G,\Om)$, among all the many torsion free symplectic left invariant connections there is  
one which depends only on the structure of Lie group and $\Om$. This connection appeared first in \cite{bou}. Indeed, 
let $(G,\Om)$ be a symplectic Lie group $(G,\Omega)$, i.e.,   a Lie group $G$ endowed with a left invariant symplectic form $\Omega$. Its Lie algebra $\G$ endowed by $\omega=\Omega(e)$ is called a symplectic Lie algebra and denoted by $(\G,\omega)$.  The product $\bullet$  on $\G$ defined by 
\begin{equation}\label{products}
\omega\left(u\bullet v,w\right)=\frac{1}{3}\left(\omega\left([u,v],w\right)+\omega\left([u,w],v\right)\right),
\end{equation}   for any $u,v,w\in\G$, extends to a left invariant connection $\na$ on $G$ which is torsion free and symplectic.
As one can see from \eqref{products}, $\na$ depends only on the structure of Lie group and $\Om$ and can be thought of as the Levi-Civita connection of a symplectic Lie group. In this paper, we study the case where this connection is of vanishing curvature. In this case, we call $(G,\Omega)$ a {\it flat symplectic Lie group} and $(\G,\omega)$ a {\it flat symplectic Lie algebra}.    

Recall that if $(G,\Om)$ is a symplectic Lie group, then the left-invariant connection $\overline{\na}$ defined by
\begin{equation}\label{natural}
\om\left(\overline{\na}_uv,w\right)=\om\left(v,[w,u]\right), \mbox{ for any }u,v,w\in\G,  
\end{equation}
is of vanishing curvature. Note that $\overline{\na}$ is symplectic ($\overline{\na}\Om=0$) if and only if $G$ is abelian.

In this paper, we show that if $(\G,\omega)$ is a flat symplectic Lie algebra then $Z(\G)$ and $[\G,\G]$ must be degenerate with respect to $\om$. We adapt the process of {\it double extension} developed in \cite{RevoyMedina, Aubert, DardieMedina} to our context and we show that all flat symplectic Lie algebras can be obtained by this process. Using this method, we prove that a flat symplectic Lie group must be nilpotent. Thus, the connection $\na$ in a flat symplectic Lie group is always geodesically complete. In particular, we deduce that flat symplectic Lie algebras are obtained by a sequence of double extension of flat symplectic Lie algebras starting from $\{0\}$. Finally, as applications, we classify all flat symplectic Lie algebras of dimension $\leq6$.

The paper contains five sections. In Section \ref{section2}, we give some general results on flat symplectic Lie algebras, and in particular, we prove that $[\G,\G]$ is degenerate in this case. We describe the method of double extension of flat symplectic Lie algebras in Section \ref{section3}. In Section \ref{section4}, we prove that every flat symplectic Lie algebra is nilpotent with degenerate center. We also show in this section that the double extension process characterize all flat symplectic Lie algebras. In Section \ref{section5}, we show that, in dimension 4, there exists only one non-abelian flat symplectic Lie algebra, and we classify up to a Lie isomorphism, all 6-dimensional flat symplectic Lie algebras.

Throughout this paper, $\G$ is a finite dimensional Lie algebra over the real field $\mathbb{R}$. A symplectic vector space $(V,\omega)$ is a real finite dimensional vector space $V$ endowed with a non-degenerate bilinear skew-symmetric form $\omega$. A symplectic basis of a symplectic vector space $(V,\omega)$ is a basis $\{e_1,\ldots,e_{2n}\}$ of $V$ such that $\omega(e_i,e_{i+1})=-\omega(e_{i+1},e_i)=1$, $i=1,\ldots,2n-1$, and all others are zero. Let $F$ be a subspace of $(V,\omega)$. $F^\bot$ is the subspace defined by $F^\bot=\{x\in\G/\omega(x,y)=0\mbox{ for any } y\in F\}$. $F$ is said to be degenerate (resp. non-degenerate, totally isotropic, lagrangian) if $F\cap F^\bot\neq\{0\}$ (resp. $F\cap F^\bot=\{0\}$, $F\subset F^\bot$, $F=F^\bot$). If $f:V\longrightarrow V$ then $f^*$ is the endomorphism of $V$ defined by $\omega(f(x),y)=\omega(x,f^*(y))$ for any $x,y\in V$.  
\section{Preliminaries} \label{section2}
Let $(G,\Om)$ be a symplectic Lie group, i.e., a Lie group endowed with a left invariant symplectic form. 
We consider the associated symplectic Lie algebra $(\G,\omega)$. For any $u\in\G$, we denote by $u^+$ the left invariant vector field on $G$ associated to $u$. The condition $d\Om=0$ is equivalent to   ,
\begin{equation}\label{sym}
\omega\left([u,v],w\right)+\omega\left([v,w],u\right)+\omega\left([w,u],v\right)=0,
\end{equation}for any $u,v,w\in\G$. The product $\bullet$ given by \eqref{products} 
defines a left invariant connection $\na$ on $G$ given by
\[ \na_{u^+}v^+=(u\bullet v)^+,\quad u,v\in\G. \]
 Note that \eqref{products} can be written
\begin{equation}\label{L} \mathrm{L}_u=\frac13\left(\ad_u-\ad_u^*   \right)  \end{equation}
where $\mathrm{L}_uv=u\bullet v$. This implies that $\Ll_u$ is skew-symmetric with respect to $\om$, i.e., 
\begin{equation}\label{dom} \om(u\bullet v,w)+\om(v,u\bullet w)=0 \end{equation}for any $u,v,w\in\G$. This is equivalent to $\na\Om=0$.  Moreover, by using \eqref{sym}, we can deduce that $\bullet$  is Lie admissible, i.e., $[u,v]=u\bullet v-v\bullet u$, for any $u,v\in\G$. This can be written $\ad_u=\Ll_u-\Rr_u$ where $\ad_uv=[u,v]$ and $\mathrm{R}_uv=v\bullet u$. At the Lie group level this is equivalent to $\na$ is torsion free.

Let ${Z}(\G)$ be the center of the Lie algebra $\G$,  $[\G,\G]$ its derived ideal,   $\G\G=\mathrm{span}\{u\bullet v/u,v\in\G\}$,  $N^\ell(\G)=\{u\in\G,\Ll_u=0\}$ and 
$N^r(\G)=\{u\in\G,\Rr_u=0\}$. It is obvious from \eqref{L} that $N^\ell(\G)=\{u\in\G,\ad_u^*=\ad_u\}$.
\begin{pr}\label{bot}
Let $(\G,\om)$ be a symplectic Lie algebra. Then
\[
[\G,\G]^\bot=\{u\in\G/\ad_u+\ad_u^*=0\}\esp Z(\G)= (\G\G)^\bot=N^\ell(\G)\cap N^r(\G)=
N^\ell(\G)\cap[\G,\G]^\perp.
\]
If $\G$ is not solvable, then the solvable radical $R(\G)$ is degenerate with respect to $\om$.
	
\end{pr}
\begin{proof} The first relation is an immediate consequence of \eqref{sym}. The equality $(\G\G)^\bot=\{u\in\G/\Rr_u=0\}$ follows immediately from \eqref{dom}. Moreover,
	by virtue of \eqref{L} if $u\in Z(\G)$ then $\Ll_u=0$ and hence $\Rr_u=0$. Suppose now that $\Rr_u=0$. Then $\ad_u=\Ll_u$ and hence $\ad_u^*=-\ad_u$. But $\Ll_u=\frac13(\ad_u-\ad_u^*)$ and hence $\ad_u=-\frac12\ad_u^*$. It follows that $\ad_u=0$. If $\G$ is not solvable and $R(\G)$ is nondegenerate then, according to \eqref{natural}, $R(\G)^\perp$ is stable by the left symmetric product induced by $\overline{\nabla}$. But $R(\G)^\perp$ is isomorphic to $\G/R(\G)$ which is semi-simple. This is a contradiction, since the Lie algebra associated to a left symmetric product cannot be semi-simple \cite{Hel}.
\end{proof}
Recall that $\G$ is unimodular ($G$ is unimodular) if, for any $u\in\G$, $\tr(\ad_u)=0$. Let $H$ be the vector defined by $\omega(H,u)=\tr(\ad_u)$ for any $u\in\G$.  Then $\G$ is unimodular if and only if $H=0$.

A vector subspace $I$ is called a left ideal (resp. right ideal) if $\G\bullet I\subset I$ (resp. $I\bullet \G\subset I$). It is called two-sided ideal if it is left and right ideal. We call $I$ a Lie ideal if $[\G,I]\subset I$. It is obvious that a two-sided ideal is a Lie ideal.

\begin{pr}\label{ideal} If $I$ is a Lie ideal of $\G$ then
	\[ I^\perp\bullet I\subset I,\;\; I\bullet I^\perp\subset I\esp I^\perp\bullet I^\perp\subset I^\perp. \]In particular, $I^\perp$ is a Lie subalgebra.

	\end{pr}
	\begin{proof} From \eqref{sym}, we have for any $u,v\in I^\perp$ and $w\in I$, we get that $\om([u,v],w)=0$ and hence $I^\perp$ is a Lie subalgebra. Moreover, for any $u, v\in I^\perp$ and any $w\in I$, by using \eqref{products}, we get that $\om(u\bullet v,w)=0$ and hence $I^\perp\bullet I^\perp\subset I^\perp$. Since the $\Ll_u$ are skew-symmetric, we deduce that $I^\perp\bullet I\subset I$. On the other hand, for any $u\in I$ and $v,w\in I^\perp$, by virtue of \eqref{products}, $\om(u\bullet v,w)=0$ and hence $I\bullet I^\perp\subset I$.
		\end{proof}

The symplectic Lie group $(G,\Om)$ is flat if the curvature $K$ of $\na$ vanishes identically.  This is equivalent to  
 \begin{equation}\label{flat} \mathrm{L}_{[u,v]}=[\mathrm{L}_u,\mathrm{L}_v] \end{equation} for any $u,v\in\G$. One can see easily that \eqref{flat} is equivalent to 
 \begin{equation}\label{flatbis}
 \Rr_{u\bullet v}-\Rr_v\circ\Rr_u=\left[\Ll_u,\Rr_v\right],
 \end{equation}{ for any } $u,v\in\G$. In this case we call $(\G,\omega)$  a flat symplectic Lie algebra. 
 
 We give here, a simple criteria to show that a symplectic Lie algebra $(\G,\omega)$ is flat.
 \begin{pr}
 Let $(\G,\omega)$ be a symplectic Lie algebra. If $Z(\G)^\bot\subset Z(\G)$ (in particular if $Z(\G)$ is lagrangian), then $(\G,\omega)$ is flat.	In this case $\bullet$ is associative and $(\G,[\ ,\ ])$ is a two-step nilpotent Lie algebra. 
 \end{pr}
  \begin{proof}
  	If $Z(\G)^\bot\subset Z(\G)$, then from $Z(\G)=(\G\G)^\bot$, we deduce that  $(\G\G)\subset Z(\G)$. Since $Z(\G)\subset N^\ell(\G)$, then $(u\bullet v)\bullet w=u\bullet(v\bullet w)=0$ for any $u,v,w\in\G$. Thus $\bullet$ is an associative product and $(\G,\omega)$ is flat. Since $[u,v]=u\bullet v-v\bullet u$, then $[\G,\G]\subset Z(\G)$ and $(\G,[\ ,\ ])$ is two-step nilpotent in this case.
  \end{proof}
 \begin{exem}
 	Let $\G^2_6$ the 6-dimensional two-step nilpotent Lie algebra defined by the only non-vanishing Lie brackets $[x_1,x_2]=x_5$ and $[x_1,x_3]=x_6$. It is proved in page 52 of \cite{khakim} that $\G^2_6$ admits three non-isometric symplectic forms: $\omega_1=x_1^*\wedge x_6^*+x_2^*\wedge x_5^*+x_3^*\wedge x_4^*$,  $\omega_2=x_1^*\wedge x_4^*+x_2^*\wedge x_6^*+x_3^*\wedge x_5^*$ and $\omega_3=x_1^*\wedge x_6^*+x_2^*\wedge x_5^*-x_3^*\wedge x_4^*$. Since $Z(\G)$ is lagrangian with respect to $\omega_1$, $\omega_2$ and $\omega_3$, then every symplectic form on $\G^2_6$ is flat.
 \end{exem}
\begin{pr}\label{prbot}
Let $(\G,\omega)$ be a flat symplectic Lie algebra. Then 
\begin{enumerate}
	\item For any $u,v\in[\G,\G]^\bot$, $u\bullet v=0$ and $\ad_u\circ\ad_v=0$. In particular, $[\G,\G]^\bot$ is abelian and, for any $u\in[\G,\G]^\perp$, $\ad_u^2=0$.
	\item $H\in[\G,\G]\cap[\G,\G]^\bot$.
\end{enumerate}	
\end{pr}
\begin{proof}
\begin{enumerate}
	\item Let $u,v\in[\G,\G]^\bot$. By virtue of \eqref{products},  one has for any $w\in\G$,
	\begin{eqnarray*}
	\omega(u\bullet v,w)&=&\frac{1}{3}\omega\left([u,v],w\right)+\frac{1}{3}\omega\left([u,w],v\right),\\
	&=&\frac{1}{3}\omega\left([u,v],w\right),\\
	&\stackrel{\eqref{sym}}=&-\frac{1}{3}\omega\left([v,w],u\right)
	-\frac{1}{3}\omega\left([w,u],v\right),\\
	&=&0.
	\end{eqnarray*}
Then $u\bullet v=0$ and $[u,v]=u\bullet v-v\bullet u=0$. 

On the other hand, from \eqref{sym} and the fact that $\ad_u=\mathrm{L}_u-\mathrm{R}_u$, one can deduce easily that $\Ll_u=\frac{2}{3}\ad_u$ and $\Rr_u=-\frac{1}{3}\ad_u$ for any $u\in[\G,\G]^\perp$.  Now, for $u,v\in[\G,\G]^\bot$, and since $u\bullet v=0$, the relation \eqref{flatbis} implies that   $-\frac{1}{9}○\ad_u\circ\ad_v=-\frac{2}{3}\ad_{[v,u]}$, which implies that $\ad_u\circ\ad_v=0$ and in particular $\ad_u^2=0$.
\item For any $u,v\in\G$, one has $\omega\left(H,[u,v]\right)=\tr\left(\ad_{[u,v]}\right)=0$ and hence $H\in[\G,\G]^\bot$. Let $u\in[\G,\G]^\bot$. Since $\ad_u$ is nilpotent then $\omega(H,u)=\tr\left(\ad_u\right)=0$. It follows that $H\in[\G,\G]\cap[\G,\G]^\bot$ as desired.
\end{enumerate}	
\end{proof}

It is known that there is a correspondence between connected and simply connected affine Lie groups and left symmetric algebras. An affine Lie group is a Lie group endowed with a torsion free flat left invariant connection.
A left symmetric algebra is an algebra $(A,.)$ satisfying
\[
\mathrm{ass}(u,v,w)=\mathrm{ass}(v,u,w),
\]
where $\mathrm{ass}(u,v,w)=(u.v).z-u.(v.w)$, for any $u,v,w\in\G$. It is known that 
a left symmetric algebra is Lie-admissible.  Let $G(A)$ be a  Lie group whose Lie algebra is $(A,[\;,\;])$ where $[u,v]=u.v-v.u$ and denote by $u^+$ the left invariant vector field on $G(A)$ associated to $u\in A$.  Then the left invariant connection $\na$ on $G(A)$ given by $\na_{u^+}v^+=(u.v)^+$ is torsion free and has vanishing curvature. Moreover, $\na$ is geodesically complete 
  if and only if $\Rr_u$ is nilpotent for any $u\in A$ \cite{Segal}. This is also equivalent to $\tr(\Rr_u)=0$ for any $u\in A$. It is known that if $\na$  is complete then $(A,[\;,\;])$ is solvable and hence $G(A)$ is solvable \cite{Hel}.
  
  Let $(G,\Om)$ be a flat symplectic Lie group and $(\G,\om)$ its associated symplectic Lie algebra. Then the vanishing of the curvature of $\Om$ is equivalent to $(\G,\bullet)$ is a left symmetric algebra. Since for any $u\in\G$, $\Ll_u$ is skew-symmetric then $\tr(\Rr_u)=-\tr(\ad_u)$ and hence $\na$ is complete if and only if $G$ is unimodular and we get the following result.

\begin{pr}\label{unimodular}
	
	The symplectic connection $\na$ of a flat symplectic Lie group is complete if and only if the group is unimodular and in this case it is solvable. 
	
	\end{pr}

\begin{pr}\label{degenerate}
Let $(\G,\omega)$ be a flat symplectic Lie algebra. Then $[\G,\G]$ is degenerate.
\end{pr}
\begin{proof}
Suppose in the contrary that $[\G,\G]$ is nondegenerate. Then $\G=[\G,\G]\oplus[\G,\G]^\bot$. Since $H\in[\G,\G]\cap[\G,\G]^\bot$ then $\G$ must be unimodular and hence $\G$ is solvable. Thus $[\G,\G]$ is nilpotent and hence its center $Z([\G,\G])\not=\{0\}$. Let us show that $Z([\G,\G])\subset Z(\G)$ which is a contradiction since, according to Proposition \ref{bot}, $Z(\G)\subset [\G,\G]^\perp$. Indeed, if $u\in Z([\G,\G])$ then for any $v\in [\G,\G]^\perp$, we have for any $w\in[\G,\G]$,
\[ [[u,v],w]=[u,[v,w]]+[v,[w,u]]=0. \]On the other hand, for any $w\in[\G,\G]^\perp$, according to Proposition \ref{prbot}, $\ad_v\circ\ad_w=0$ and hence $[[u,v],w]=0$. Thus $[u,v]\in Z(\G)\subset [\G,\G]^\perp$ and hence $[u,v]=0$ which completes the proof.
\end{proof}

\begin{pr}\label{Lemma2}
	Let $(\G,\omega)$ be a flat symplectic Lie algebra. Then
	 $N^\ell(\G)$ is a two-sided ideal and in particular, it is a Lie ideal.
		Moreover, $[\G, [\G,\G]^\bot ]\subset  N^\ell(\G)$.
	
\end{pr}
\begin{proof}
	 It is trivial that $N^\ell(\G)$ is a right ideal. Let us show that it is also a left ideal. Let $u\in N^\ell(\G)$ and $v,w\in\G$. Since $\bullet$ is left symmetric, we have
		\[
		(v\bullet u)\bullet w-v\bullet (u\bullet w)=(u\bullet v)\bullet w-u\bullet (v\bullet w),
		\]
		we deduce that $(v\bullet u)\bullet w=0$ and hence $v\bullet u\in N^\ell(\G)$.

		 Let $u,v\in\G$. From \eqref{products} one has, $3\Ll_u=\ad_u-\ad_u^*$. Thus
		\begin{eqnarray}\nonumber
		0&=&\Ll_{[u,v]}-\left[\Ll_u,\Ll_v\right]\\\nonumber
		&=&\frac{1}{3}\left(\ad_{[u,v]}-\ad^*_{[u,v]}\right)-\frac{1}{9}\left[\ad_u-\ad_u^*,\ad_v-\ad^*_v\right]\\\label{curvature}
		&=&\frac{2}{9}\ad_{[u,v]}-\frac{2}{9}\ad^*_{[u,v]}+\frac{1}{9}\left[\ad_u^*,\ad_v\right]+\frac{1}{9}\left[\ad_u,\ad^*_v\right].
		\end{eqnarray} 
		Now for any $u\in\G$ and $v\in[\G,\G]^\perp$, we have $\ad_v^*=-\ad_v$ and if we replace in the relation above, we get
		\begin{eqnarray*}
		0	&=&\frac{2}{9}\ad_{[u,v]}-\frac{2}{9}\ad^*_{[u,v]}+\frac{1}{9}\left[\ad_u^*,\ad_v\right]+\frac{1}{9}\left[\ad_u,\ad^*_v\right]\\
			&=&\frac{1}{9}\ad_{[u,v]}-\frac{1}{9}\ad^*_{[u,v]}.
		\end{eqnarray*}
		Thus $\ad_{[u,v]}=\ad_{[u,v]}^*$ 	and hence  $[u,v]\in N^\ell(\G)$.		
\end{proof}
The symplectic reduction is an important tool in the study of symplectic Lie algebras (see \cite{cortes}). If $(\G,\om)$ is a symplectic Lie algebra and $I\subset \G$ is a totally isotropic Lie ideal, then $I^\perp/I$ carries a natural symplectic structure called the symplectic reduction of $(\G,\om)$ by $I$. The following proposition shows that the symplectic reduction of a flat symplectic Lie algebra is also flat. A symplectic Lie algebra is called completely reducible if it admits a sequence of symplectic reduction to the trivial symplectic Lie algebra. Otherwise, it is called irreducible.

\begin{pr} Let $(\G,\om)$ be a flat symplectic Lie algebra and $I$ a Lie ideal. Then the symplectic Lie algebra structure on $I^\perp/I\cap I^\perp$ is flat.
	
	\end{pr} 
\begin{proof} According to Proposition \ref{ideal}, $I^\perp$ is stable by $\bullet$ and $I\cap I^\perp$ is a two-sided ideal of $I^\perp$. Hence $\bullet$ defines a product $\wi\bullet$ which is the product associated to the symplectic Lie algebra structure on $I^\perp/I\cap I^\perp$.
	\end{proof}
	
	\begin{exem} It is known that the Lie algebra $\mathrm{aff}(n,\R)$ of the affine group has a symplectic Lie algebra structure (see \cite{bordeman}). Moreover, Cort\'es in \cite{cortes} showed that $\mathrm{aff}(n,\R)$ has a symplectic reduction to $\mathrm{aff}(n-1,\R)$. It is easy to check that the symplectic structure of $\mathrm{aff}(1,\R)$ is not flat (see example \ref{dim2}). Then the symplectic structure of 
		$\mathrm{aff}(n,\R)$ is not flat. 
		
		\end{exem}

		\begin{pr} A flat symplectic Lie algebra is completely reducible.

			\end{pr}

\section{Double extension of flat symplectic Lie algebras}\label{section3}
The double extension method constitute a powerful tool to study Lie groups endowed with left-invariant structures (see for instance \cite{Aubert, RevoyMedina, DardieMedina}). In \cite{Aubert}, the authors developed the double extension method to study Lie groups endowed with flat left-invariant pseudo-Riemannian metrics. In \cite{colombien}, the author adapted this process to study Lie groups endowed by flat symplectic connection. In this section, we adapt this method to our cennection given by \eqref{products}, and we will show in the next section, that all flat symplectic Lie algebras are obtained by this method.

Let $(B,\omega_B)$ be a flat symplectic Lie algebra. We denote by $ab$ the product defined by \eqref{products} for any $a,b\in B$. Let $\xi\in \End(B)$ and $b_0\in B$ such that for any $a,b\in B$
\begin{eqnarray}\label{eq1}
\left[\xi,\xi^*\right]&=&\xi^2-\frac{1}{3}\Rr_{b_0},\\\label{eq2}
b_0&\in&\ker\left(\xi^*-\xi\right),\\\label{eq3}
\xi^*\circ\xi&=&\frac{1}{3}\left(\Rr_{b_0}+\Rr_{b_0}^*\right),\\\label{eq4}
\xi\left([a,b]\right)&=&a\xi(b)-b\xi(a),\\\label{eq5}
\xi^*(ab)-\xi^*(a)b-a\xi^*(b)&=&\xi(ab)-a\xi(b)-2\xi(a)b.	
\end{eqnarray}
A couple $(\xi,b_0)\in \End(B)\times B$ which satisfies the equations \eqref{eq1}-\eqref{eq5} is called {\it admissible}. Let $\G$ be the vector space  defined by $\G=\mathbb{R}e\oplus B\oplus\mathbb{R}\bar{e}$ endowed with the non-degenerate skew-symmetric form $\omega$ defined by $\omega_{/B\times B}=\omega_B$, $\omega(e,e)=\omega(\bar{e},\bar{e})=\omega(e,B)=\omega(\bar{e},B)=0$ and $\omega(e,\bar{e})=1$. We define also in $\G$ the Lie brackets 
\begin{equation}\label{Liebrackets}
\left[\bar{e},a\right]=\left(\xi^*-2\xi\right)(a)+\omega_B(b_0,a)e\esp [a,b]=[a,b]_B+\omega_B\left((\xi+\xi^*)(a),b\right)e.
\end{equation}
Let us show that $(\G,\omega)$ is also a flat symplectic Lie algebra. First, one can show easily that $\om$ satisfies \eqref{sym}.
Thus, the product defined by \eqref{products} satisfies $[u,v]=u\bullet v-v\bullet u$ for any $u,v\in\G$. Let us show that this product is left symmetric. From \eqref{products}, one has, for any $a,b\in B$,
\begin{eqnarray*}
\Ll_e&=&\Rr_e=0,\\
\bar{e}\bullet a&=&\left(\xi^*-\xi\right)(a)+\frac{1}{3}\omega_B(b_0,a)e,\\
a\bullet\bar{e}&=&\xi(a)-\frac{2}{3}\omega_B(b_0,a)e,\\
a\bullet b&=&ab+\omega_B\left(\xi(a),b\right)e,\\
\bar{e}\bullet\bar{e}&=&\frac{1}{3}b_0,	
\end{eqnarray*}
where the product of $a,b\in B$ is denoted in $B$ by $ab$ and in $\G$ by $a\bullet b$. Since $\Ll_e=\Rr_e=0$ then $\ass(e,u,v)=\ass(u,e,v)=\ass(u,v,e)=0$ for any $u,v\in\G$. Let $a,b,c\in B$. The equation $\ass(\bar{e},a,\bar{e})=\ass(a,\bar{e},\bar{e})$ is equivalent to \eqref{eq1} and \eqref{eq2}. The equation $\ass(\bar{e},a,b)=\ass(a,\bar{e},b)$ is equivalent to \eqref{eq1} and \eqref{eq5}. The equation $\ass(a,b,\bar{e})=\ass(b,a,\bar{e})$ is equivalent to \eqref{eq3} and \eqref{eq4}. Finally, The equation $\ass(a,b,c)=\ass(b,a,c)$ is equivalent to \eqref{eq4} and to the fact that $(B,\omega_B)$ is a flat symplectic Lie algebra. In summary, if $(B,\omega_B)$ is a flat symplectic Lie algebra of dimension $2n$ and $(\xi,b_0)\in \End(B)\times B$ such that \eqref{eq1}-\eqref{eq5} are satisfied, then $(\G,\omega)$ is a flat symplectic Lie algebra of dimension $2n+2$.
\begin{Def}
The flat symplectic Lie algebra $(\G,\omega)$ constructed as above is called the  double extension of $(B,\omega_B)$ by means of $\xi$ and $b_0$.  	
\end{Def}
Conversely, we have the following result.
\begin{pr}\label{inversede}
Let $(\G,\omega)$ be a flat symplectic Lie algebra which admits a one-dimensional two-sided ideal $I$ (with respect to the product given by \eqref{products}) such that $I^\bot$ is also a two-sided ideal. Then $(\G,\omega)$ is a double extension of another flat symplectic Lie algebra of dimension $\dim\G-2$ by means of $\xi$ and $b_0$.  		
\end{pr}
\begin{proof}
Let $(\G,\omega)$ be a flat symplectic Lie algebra which admits a one-dimensional two-sided ideal $I$. Since $I^\bot$ is also a two-sided ideal, then $I^\bot/I$ can be endowed by a flat symplectic structure. In fact, for any $\bar{x},\bar{y}\in I^\bot/I$, we put $\bar{\omega}(\bar{x},\bar{y})=\omega(x,y)$. Then $\bar{\omega}$ is a symplectic structure on $I^\bot/I$. Furthermore, the product defined on $I^\bot/I$ by 
\[
\bar{\omega}\left(\bar{x}\bar{y},\bar{z}\right)=\frac{1}{3}\left(\bar{\omega}\left([\bar{x},\bar{y}],\bar{z}\right)+\bar{\omega}\left([\bar{x},\bar{z}],\bar{y}\right)\right),
\]
is left symmetric. Thus $\left(I^\bot/I,\bar{\omega}\right)$ is a flat symplectic Lie algebra. We put $I=\mathbb{R}e$ and let $\bar{e}\in\G$ such that $\omega(e,\bar{e})=1$. Let $B'=\{e,\bar{e}\}^\bot$. Then $\G=\mathbb{R}e\oplus B'\oplus\mathbb{R}\bar{e}$ where $I^\bot=\mathbb{R}e\oplus B'$. Since $I^\bot$ is a two-sided ideal, then for any $a,b\in B'$ one has 
\[
(\alpha e+a)\bullet(\beta e+b)=f(a,b)e+a\star b,
\]
where $a\star b$ is the component of $a\bullet b$ over $B'$. The product on $I^\bot$ is left symmetric is equivalent to the fact that $(B',\star)$ is a left symmetric algebra and 
\begin{equation}\label{2cocycle}
f\left([a,b]_{B'},c\right)=f(a,b\star c)-f(b,a\star c),\mbox{ for any }a,b,c\in B'.
\end{equation}
The equation \eqref{2cocycle} is equivalent to the fact that $f$ is a scalar 2-cocycle for the left symmetric algebra $(B',\star)$ (see \cite{Nijenhuis} \cite{Aubert}). Let $H^2_{SG}(B',\mathbb{R})$ be the second space of scalar cohomology of the left symmetric algebra $B'$ and $H^1_L(B',B')$ the first space of cohomology of the Lie algebra $B'$ with respect to the representation $\Ll$. It is proved in \cite{Aubert} that $H^2_{SG}(B',\mathbb{R})$ is isomorphic to $H^1_L(B',B')$ via the equation 
\[
f(a,b)=\omega\left(\xi(a),b\right) \mbox{ for any }a,b\in B',
\]
where $\xi:B'\longrightarrow B'$ such that $\xi\left([a,b]_{B'}\right)=a\star\xi(b)-b\star\xi(a)$. On the other hand, the application $\phi:B'\longrightarrow I^\bot/I$ defined by $\phi(a)=\bar{a}$ for any $a\in B'$ is an isomorphism of left symmetric algebra. Thus, we can identify $B'$ with $B=I^\bot/I$ and hence $\G=\mathbb{R}e\oplus B\oplus\mathbb{R}\bar{e}$ where $(B,\omega_B)=\left(I^\bot/I,\bar{\omega}\right)$ is a flat symplectic Lie algebra. The product on $I^\bot=\mathbb{R}e\oplus B$ is given by 
\[
(\alpha e+a)\bullet(\beta e+b)=\omega_B\left(\xi(a),b\right)e+ab.
\] 
Thus, the Lie brackets on $\G$ are given by, for any $a,b\in B$
\[
\left[\bar{e},e\right]=\lambda e,\ \left[\bar{e},a\right]=D(a)+\omega_B(b_0,a)e\mbox{ and }[a,b]=[a,b]_B+\omega_B\left((\xi+\xi^*)(a),b\right)e,
\]
where $D\in \End(B)$, $\lambda\in\mathbb{R}$ and $b_0\in B$. We have $e\bullet\bar{e}=-\frac{2}{3}\lambda e$, $\bar{e}\bullet e=\frac{\lambda}{3}e$ and $\bar{e}\bullet\bar{e}=\frac{1}{3}b_0+\frac{\lambda}{3}\bar{e}$. Since $(\G,\omega)$ is flat then 
\[
(e\bullet\bar{e})\bullet\bar{e}-e\bullet(\bar{e}\bullet\bar{e})=(\bar{e}\bullet e)\bullet\bar{e}-\bar{e}\bullet(e\bullet\bar{e}),
\]
which implies that $\lambda=0$. From, 
\[
\omega(a\bullet b,\bar{e})=\frac{1}{3}\omega_B\left(\xi+\xi^*-D(a),b\right),
\]
and $\omega(a\bullet b,\bar{e})=\omega_B\left(\xi(a),b\right)$, one can deduce that $D=\xi^*-2\xi$. Therefore, the Lie brackets in $\G=\mathbb{R}e\oplus B\oplus\mathbb{R}\bar{e}$ are reduced to 
\[
\left[\bar{e},a\right]=\left(\xi^*-2\xi\right)(a)+\omega_B(b_0,a)e\mbox{ and }[a,b]=[a,b]_B+\omega_B\left((\xi+\xi^*)(a),b\right)e.
\]
As in the previous paragraph, one can show that $(\G,\omega)$ is a flat symplectic Lie algebra if and only if \eqref{eq1}-\eqref{eq5} hold. Thus $(\G,\omega)$ is a double extension of a flat symplectic Lie algebra $(B,\omega_B)$ by means of $\xi$ and $b_0$.
\end{proof}
\begin{rem}
	\begin{itemize}
		\item As we have shown in Proposition \ref{inversede}, if $I=\mathbb{R}e$ is a two-sided ideal of a flat symplectic Lie algebra, then $e\in Z(\G)$.
		\item If $(\G,\omega)$ is a double extension of $(B,\omega_B)$ by means of $\xi$ and $b_0$, then $\xi^*-2\xi$ is a derivation of the Lie algebra $B$. Indeed, according to \eqref{eq5}, one can check that for any $a,b\in B$,
		\[
		\xi^*-2\xi\left([a,b]\right)=\left[\xi^*-2\xi(a),b\right]+\left[a,\xi^*-2\xi(b)\right],
		\]
	\end{itemize}
\end{rem}
\begin{pr}\label{nipotentinverse}
	Let $(B,\omega_B)$ be a flat symplectic Lie algebra and $(\G,\omega)$ the double extension of $(B,\omega_B)$ by means of $(\xi,b_0)$. If $B$ is nilpotent then $\G$ is nilpotent.
\end{pr}
\begin{proof}
Let $(B,\omega_B)$ be a flat symplectic Lie algebra and let $(\xi,b_0)$ be an admissible couple of $\End(B)\times B$. Let us show that, for any $k\in\mathbb{N}^*$, and $a\in B$,
\begin{equation}\label{ee1}
\tr\left(\xi^k\circ\Rr_a\right)=\tr\left(\Rr_{\xi^k(a)}\right).
\end{equation} 
Indeed, for any $a,b\in B$, $\xi\left([b,a]\right)=b.\xi(a)-a.\xi(b)$, thus $\xi\circ\Rr_a-\xi\circ\Ll_a=\Rr_{\xi(a)}-\Ll_a\circ\xi$, and hence $\tr\left(\xi\circ\Rr_a\right)=\tr\left(\Rr_{\xi(a)}\right)$. Then the equation is true for $k=1$. We suppose the property is true for $k$ and let us show that it is also correct for $k+1$. We have
\[
\xi^{k+1}\left([a,b]\right)=\xi^{k+1}\circ\Ll_a(b)-\xi^{k+1}\circ\Rr_a(b).
\] 
On the other hand,
\[
\xi^{k+1}\left([a,b]\right)=\xi^{k}\circ\xi\left([a,b]\right)=\xi^{k}\left(a.\xi(b)-b.\xi(a)\right)=\xi^{k}\circ\Ll_a\circ\xi(b)-\xi^{k}\circ\Rr_{\xi(a)}(b).
\] 
It follows that $\tr\left(\xi^{k+1}\circ\Rr_a\right)=\tr\left(\xi^k\circ\Rr_{\xi(a)}\right)=\tr\left(\Rr_{\xi^{k+1}(a)}\right)$.\\
We put $D=\xi^*-2\xi$. From \eqref{eq1}, one has $[\xi,D]=\xi^2-\frac{1}{3}\Rr_{b_0}$. By induction, one can show that for any $k\in\mathbb{N}^*$,
\[
\left[\xi^k,D\right]=k\xi^{k+1}-\frac{1}{3}\sum_{p=0}^{k-1}\xi^p\circ\Rr_{b_0}\circ\xi^{k-1-p}.
\]
Thus $\tr\left(\xi^{k+1}\right)=\frac{1}{3}\tr\left(\Rr_{b_0}\circ\xi^{k-1}\right)$, and from \eqref{ee1}, one has $\tr\left(\xi^{k+1}\right)=\frac{1}{3}\tr\left(\Rr_{\xi^{k-1}(b_0)}\right)$, for any $k\in\mathbb{N}^*$. If $B$ is nilpotent then $B$ is a complete left symmetric algebra which implies that $\tr\left(\Rr_a\right)=0$ for any $a\in B$. It follows that  $\tr\left(\xi^{k+1}\right)=0$ for any $k\in\mathbb{N}^*$ and hence $\xi$ is nilpotent. From \eqref{eq1} and \eqref{eq3}, one can deduce that $\xi\circ\xi^*=\xi^2+\frac{1}{3}\Rr_{b_0}^*$ and $\left(\xi^*\right)^2=\left[\xi,\xi^*\right]+\frac{1}{3}\Rr_{b_0}^*$. We replace in the equation $D^2=4\xi^2+\left(\xi^*\right)^2-2\xi\circ\xi^*-2\xi^*\circ\xi$ and we get $D^2=3\xi^2-3\xi^*\circ\xi=-3D\circ\xi-3\xi^2$. By induction, one can deduce that for any $k\geq2$,
\[
D^k=a_kD\circ\xi^{k-1}+b_k\xi^k,
\] 
where $a_k,b_k\in\mathbb{R}$. Since $\xi$ is nilpotent, then $D$ is nilpotent. Now, if $(\G,\omega)$ is a double extension of $(B,\omega_B)$ then the Lie brackets in $\G$ is given by 
\[
[\bar{e},a]=D(a)+\omega_B\left(b_0,a\right)_Be,\ [a,b]=[a,b]_B+\omega_B\left(\xi+\xi^*(a),b\right)e.
\]
Since $B$ is nilpotent then $D$ is nilpotent and hence $\G$ is nilpotent. This completes the proof of the proposition.
\end{proof}
\section{Flat symplectic Lie groups are nilpotent}\label{section4}
In this section, we show that any flat symplectic Lie group is nilpotent. We show also that its Lie algebra is obtained by a sequence of double extension of flat symplectic Lie algebras starting from $\{0\}$.
\begin{Le}\label{solv center}
	Let $(\G,\omega)$ be a flat symplectic Lie algebra. If $\G$ is solvable, then its center is not trivial.
\end{Le}
\begin{proof} Let us start by showing the following implication
	\begin{equation}\label{imp} Z(\G)= \{0\}\;\Longrightarrow Z([\G,\G])\cap N^\ell(\G)=\{0\}. \end{equation}
	Indeed, suppose that $Z(\G)= 0$ and let $u\in Z([\G,\G])\cap N^\ell(\G)$. Then for any $w\in\G$ and any $v\in[\G,\G]$, since $\ad_u^*=\ad_u$, we have
	\[ 0=\om([u,v],w)=\om(v,[u,w]) \] and hence $[u,w]\in[\G,\G]^\perp$. But we have seen in Proposition \ref{Lemma2} that $[u,w]\in N^\ell(\G)$ thus $[u,w]\in[\G,\G]^\perp\cap N^\ell(\G)=Z(\G)=\{0\}$ which shows that $u\in Z(\G)$ and hence $u=0$.
	
	With this remark in mind let us prove our result by contradiction. Suppose that $\G$ is solvable and $Z(\G)=\{0\}$. We consider the sequence of vector subspace given by 
	\[ C_0=[\G,\G]\cap[\G,\G]^\perp\esp C_k=[[\G,\G], C_{k-1}]\quad k\geq1. \]
	By virtue of Proposition \ref{degenerate}, $C_0\not=\{0\}$ and, according to Proposition \ref{Lemma2}, $C_k\subset N^\ell(\G)$ for any $k\geq1$. Let us show first that, for any $k\geq1$, if $C_k=\{0\}$ then $C_{k-1}=\{0\}$. Suppose that $C_k=\{0\}$. Then for any $u\in C_{k-1}$, for any $v\in[\G,\G]$ and for any $w\in\G$,
	\[ 0=\om([u,v],w)=\pm\om(v,[u,w]) \]since $\ad_u^*=-\ad_u$ if $k=1$ and $\ad_u^*=\ad_u$ if $k\geq2$. So $[u,w]\in[\G,\G]^\perp\cap N^\ell(\G)=Z(\G)=\{0\}$ and hence $u=0$. This completes the proof of our assertion and since $C_0\not=\{0\}$ then for any $k\geq1$, $C_k\not=\{0\}$. But $[\G,\G]$ is nilpotent and hence there exists $k_0$  such that $C_{k_0}\subset Z([\G,\G])$. We have also $C_k\subset N^\ell(\G)$ which contradicts \eqref{imp} and hence $Z(\G)\not=\{0\}$.
	\end{proof}

\begin{exem}\label{dim2} The non abelian symplectic Lie algebra of dimension 2 is solvable and has a vanishing center so it is not flat.
\end{exem}
\begin{Le}\label{center}
	Let $(\G,\om)$ be a flat symplectic Lie algebra. Then its center is not trivial.
\end{Le}
\begin{proof}
	Let $(\G,\om)$ be a flat symplectic Lie algebra. If $\G$ is unimodular, then it is solvable, and the result follows from  lemma \ref{solv center}. Suppose $\G$ is non-unimodular which implies that $H\neq0$. We will show that $H\in Z(\G)$. Let $u\in \mathfrak{g}$, we put  $x_{0}=u$ and $x_{n+1}=L_{x_{n}}(u)$   for any $n\geq 0$. Let us show by induction that, for any $n\geq 1$ \begin{equation}
	R_{x_{n}}=R_{u}^{n+1}+R_{u}^{n-2}\circ [L_{u},R_{u}]+R_{u}^{n-3}\circ [L_{x_{1}},R_{u}]+...+[L_{x_{n-1}},R_{u}].\label{RR}\end{equation}
	From $(\ref{flatbis})$, the relation is true for $n=1$. We assume the property established at the rank $n\geq 1$. According to the relation $(\ref{flatbis})$ we have
	\begin{equation*}
	\begin{aligned}
	R_{x_{n+1}}=R_{x_{n}\cdot u}=&   R_{u}\circ R_{x_{n}}+[L_{x_{n}}, R_{u}]\\=& R_{u}\circ( R_{u}^{n+1}+R_{u}^{n-2}\circ [L_{u},R_{u}]+R_{u}^{n-3}\circ [L_{x_{1}},R_{u}]+...+[L_{x_{n-1}},R_{u}])+ [L_{x_{n}}, R_{u}] \\=& R_{u}^{n+2}+R_{u}^{n-1}\circ [L_{u},R_{u}]+R_{u}^{n-2}\circ [L_{x_{1}},R_{u}]+...+[L_{x_{n}},R_{u}]
	\end{aligned}
	\end{equation*}
	Now, according to the Proposition $\eqref{Lemma2}$, we have $[H,x]\in N_{\ell}(\mathfrak{g})$, for any $x\in\mathfrak{g}$. Moreover, let $u\in N_{\ell}^{\perp}(\mathfrak{g})$ we put, for any $n\geq 0$ $x_{n+1}=L_{x_{n}}(u)$ where $x_{0}=u$, according to the relation $(\ref{products})$, we have $$\om(H,x_{n+1})=\om(H,L_{x_{n}}(u))=\frac{1}{3}\om([H,x_{n}],u)=0.$$
	On the other hand, we have 
	$$\om(H,x_{n+1})=tr(ad_{x_{n+1}})=-tr(R_{x_{n+1}})$$ and according to the relation $\eqref{RR}$ and to the fact that $\tr\left(R_{u}^{p}\circ [L_{u},R_{u}]\right)=0$, we have, for any $n\geq 0$ $$tr(R_{x_{n+1}})=tr(R_{u}^{n+2})=0$$
	which implies that, for all $u\in N_{\ell}^{\perp}(\mathfrak{g})$, $R_{u}$  is nilpotent. It follows that $H\in N_{\ell}(\mathfrak{g}) $ and since $H\in [\G,\G]^\perp$, one has $H\in Z(\G)$ as desired.
	
\end{proof}
\begin{Le}\label{nontrivial}
A flat symplectic Lie algebra $(\G,\omega)$ is with non-trivial center if and only if 	$(\G,\omega)$ is a double extension of a flat symplectic Lie algebra $(B,\omega_B)$ by means of $(\xi,b_0)$.
\end{Le}
\begin{proof}
Let $(\G,\omega)$ be a flat symplectic Lie algebra. Assume that $Z(\G)\neq\{0\}$ and let $e$ be a non-nul vector in $Z(\G)$. Thus $\Ll_e=\Rr_e=0$ and hence $\mathrm{I}=\mathbb{R}e$ is a two-sided ideal. Furthermore, for any $a\in\mathrm{I}^\bot$ and $x\in\G$,
\[
\omega(ax,e)=-\omega(x,ae)=0\mbox{ and }\omega(xa,e)=-\omega(a,xe)=0.
\]	
Then $\mathrm{I}^\bot$ is also a two-sided ideal. Thus, according to Proposition \ref{inversede}, $(\G,\omega)$ is a double extension of a flat symplectic Lie algebra $(B,\omega_B)$ by means of $(\xi,b_0)$. Conversely, if $(\G,\omega)$ is obtained by the double extension process then according to the Lie brackets given in \eqref{Liebrackets}, one has $e\in Z(\G)$ and $\G$ has a non-trivial center as desired.
\end{proof}
\begin{theo}\label{main}
A flat symplectic Lie algebra must be nilpotent with degenerate center.
\end{theo}
\begin{proof}
By induction on $\dim \G$. If $\dim\G=2$ then according to example \ref{dim2}, $\G$ is nilpotent. Suppose the property is correct for $\dim\G=2n$, and let us show that it is also correct for $\dim\G=2n+2$. Let $(\G,\omega)$ be a flat symplectic Lie algebra of dimension $2n+2$. According to lemma \ref{center}, the center of $\G$ is not-trivial, and hence from lemma \ref{nontrivial}, $(\G,\omega)$ is a double extension of a flat symplectic Lie algebra $(B,\omega_B)$ of dimension $2n$. By the induction hypothesis, we deduce that $B$ is nilpotent and according to Proposition \ref{nipotentinverse}, $\G$ must be nilpotent. Since $Z(\G)\subset [\G,\G]^\bot$, then $Z(\G)\cap[\G,\G]\subset Z(\G)\cap Z(\G)^\bot$ and hence $Z(\G)$ is degenerate.
\end{proof}
\begin{theo}\label{th double}
	$(\G,\omega)$ is a flat symplectic Lie algebra if and only if $(\G,\omega)$ is obtained by a sequence of double extension starting from $\{0\}$. 	
\end{theo}
\begin{proof}
	Let $(\G,\omega)$ be a flat symplectic Lie algebra. Since $\G$ is nilpotent, then $Z(\G)\neq\{0\}$. Thus, according to lemma \ref{nontrivial}, $(\G,\omega)$ is a double extension of a flat symplectic Lie algebra $(B_1,\omega_{B_1})$. Since $B_1$ is nilpotent then $(B_1,\omega_{B_1})$ is also a double extension of a flat symplectic nilpotent Lie algebra $(B_2,\omega_{B_2})$. Thus, $(\G,\omega)$ is obtained by a sequence of double extension starting from $\{0\}$. 
\end{proof}
\section{Examples of flat symplectic Lie algebras of low dimensions}\label{section5}
In this section, we give all flat symplectic Lie algebras of dimension $\leq6$.
\subsection{Four dimensional flat symplectic Lie algebras}
Four dimensional flat symplectic Lie algebras are double extension of the 2-dimensional symplectic abelian Lie algebra by means of an addmissible couple $(\xi,b_0)$. Note that if $B$ is abelian then $(\xi,b_0)\in End(B)\times B$ is admissible if and only if $\xi\xi^*=\xi^2,\ \xi^*\xi=0\mbox{ and }b_0\in\ker(\xi^*-\xi)$.
\begin{pr}\label{solution admissible dim2}
	Let $(B,\om_B)$ be the symplectic  abelian Lie algebra of dimension $2$, then $(\xi, b_0)$ is admissible if and only if ($\xi=0$ and $b_0\in B$) or there exists a symplectic basis $\{e_1,e_2\}$ of $(B,\om_B)$ such that 
	\[\xi=\left(\begin{array}{cccc}0 & a  \\ 0 & 0 \end{array}\right) \text{ and }b_0=\alpha e_1,
	\]
	where $\alpha \in \mathbb{R}$ and $a\neq 0.$
\end{pr}
\begin{proof}
	Since $\xi\xi^*=\xi^2$ and $\xi^*\xi=0$ we deduce that  $\xi$ is nilpotent, then $\xi=0$ and $b_0\in B$, or $\xi^{2}=0$ and $\xi\neq0$. In the last case, there exists a symplectic basis $\{e_1,e_2\}$ of $B$ such that  
	\[
	 \xi=\left(\begin{array}{cccc}0 & a  \\ 0 & 0 \end{array}\right) \text{ and} \quad b_0=\alpha e_{1}+\beta e_2
	 \]
	where $\alpha,\beta\in \mathbb{R}$ and $a\neq0.$\\
	Form $\xi(b_{0})=\xi^*(b_{0})$
	we deduce that $\beta=0$.
\end{proof}
\begin{pr}\label{dim4}
	The only non-abelian flat symplectic Lie algebra of dimension 4 is $(\mathbb{R}\times\mathfrak{h}_3,\omega)$ where
	\[
	\mathbb{R}\times\mathfrak{h}_3:\ [x_1,x_2]=x_3\mbox{ and }\omega_0=x_1^*\wedge x_4^*+x_2^*\wedge x_3^*.
	\]	
\end{pr}
\begin{proof}
	Let $(\mathfrak{g},\om)$ be  a flat  symplectic non-abelian Lie algebra of dimension $4$. According to Theorem \ref{th double}, $\mathfrak{g}$ is a double extension of the  symplectic abelian Lie algebra of dimension $2$. According to Proposition \ref{solution admissible dim2} we have two cases:
	\begin{enumerate}
		\item[$\bullet$] If $\xi=0$ and $b_0=\alpha e_{1}+\beta e_2$ where $\alpha,\beta\in \mathbb{R}$,  then according to \eqref{Liebrackets}, there exists a symplectic basis $\{e,\bar{e},e_{1},e_2\}$ of $\mathfrak{g}$ such that the Lie brackets are given by 
		\begin{equation*}
		\left[\bar{e},e_1\right]=-\beta e,\;\left[\bar{e},e_2\right]=\alpha e.
		\end{equation*} 
		Since $\G$ is non-abelian, then $(\alpha,\beta)\neq(0,0)$. If for example $\beta\neq0$, then we put  $$(x_1,x_2,x_3,x_4)=(-\frac{1}{\beta}e_1, \bar{e}, -e, -\beta e_2-\alpha e_1).$$ 
			Thus $(\mathfrak{g},\om)$ is symplecto-isomorphic to $(\mathbb{R}\times\mathfrak{h}_3,\om_0)$.
		\item[$\bullet$] If $\xi\neq0$ and $b_0=\alpha e_{1}$ where $\alpha\in \mathbb{R}$,  then there exists a symplectic basis $\{e,\bar{e},e_{1},e_2\}$ of $\mathfrak{g}$ such that the Lie bracket is given by 
		\begin{equation*}
		\left[\bar{e},e_2\right]=\alpha e-3ae_1, \,\, a\neq 0.\end{equation*} 
		Put  $$(x_1,x_2,x_3,x_4)=\left(3a\bar{e}+\alpha e_2,\bar{e}+\frac{\alpha+1}{3a}e_2, \alpha e-3ae_1,-\frac{\alpha+1}{3a}e+e_1\right).$$ 
		Thus, in this new basis, the Lie brackets and the flat symplectic form are reduced to
		\[
		[x_1,x_2]=x_3\mbox{ and }\om=x_1^*\wedge x_4^*+x_2^*\wedge x_3^*,
		\]
		which implies that $(\mathfrak{g},\om)$ is symplecto-isomorphic to $(\mathbb{R}\times\mathfrak{h}_3,\om_0)$ as desired. 	
	\end{enumerate}
\end{proof}
\subsection{Flat symplectic Lie algebras of dimension 6}
According to Theorem \ref{th double}, flat symplectic Lie algebras of dimension 6 are obtained from 4-dimensional flat symplectic Lie algebras by using the double extension process. From Proposition \ref{dim4}, any flat symplectic Lie algebra of dimension 4 is abelian or isomorphic to $\mathbb{R}\times\mathfrak{h}_3$. Thus, we must determine all admissible couples $(\xi,b_0)$ in a flat symplectic Lie algebra $(B,\omega_B)$ where $B$ is abelian or $B$ is isomorphic to $\mathbb{R}\times\mathfrak{h}_3$.
 
{\bf First case: $B$ is abelian.\\}
In the abelian case, $(\xi,b_0)$ is admissible if and only if  
\[
\xi\xi^*=\xi^2,\ \xi^*\xi=0\mbox{ and }b_0\in\ker(\xi^*-\xi).
\]
\begin{Le}\label{Lemma1}
Let $(B,\omega_B)$ be the abelian symplectic Lie algebra of dimension 4. If $(\xi,b_0)$ is admissible then $Im(\xi)$ is totally isotropic and $\xi^2=0$. 	
\end{Le}
\begin{proof}
From $\omega\left(\xi(u),\xi(v)\right)=\omega\left(u,\xi^*\xi(v)\right)=0$, for any $u,v\in B$, we deduce that $Im(\xi)$ is totally isotropic. Let us show that $\xi^2=0$. First, we have $\xi^3=\xi\xi^*\xi=0$. Assume that $\xi^2\neq0$. Then $\ker\xi\subsetneq\ker\xi^2\subsetneq B$. Since $Im(\xi)$ is totally isotropic then $\dim Im(\xi)\leq2$. If $\dim Im(\xi)=1$ then $\dim\ker\xi=3$ and hence $\xi^2=0$, contradiction. Thus, $\dim Im(\xi)=\dim\ker\xi=2$ and $\dim\ker\xi^2=3$. We can choose a basis $\{e_1,e_2,e_3,e_4\}$ of $(B,\omega_B)$ such that $\{e_1,e_2\}$ is a basis of $\ker\xi$, $\{e_1,e_2,e_3\}$ is a basis of $\ker\xi^2$. In this basis, $\xi$ has the form
\[
\xi=\left(\begin{array}{cccc}0&0&a&c\\0&0&b&d\\0&0&0&e\\0&0&0&0\end{array}\right)\mbox{ and }\xi^2=\left(\begin{array}{cccc}0&0&0&ae\\0&0&0&be\\0&0&0&0\\0&0&0&0\end{array}\right)
\]
If $\ker\xi$ is symplectic, then we can choose the basis $\{e_1,e_2,e_3,e_4\}$ to be a symplectic basis. Thus 
\[
\xi^*=\left(\begin{array}{cccc}0&0&0&0\\0&0&0&0\\d&-c&0&-e\\-b&a&0&0\end{array}\right).
\]
From $\xi\xi^*=\xi^2$, we deduce that $ae=be=0$ which contradicts the fact that $\xi^2\neq0$. If $\ker\xi$ is totally isotropic then we can choose the basis such that $\{e_1,e_4,e_2,e_3\}$ is symplectic. In this case $\xi^*$ has the form 
\[
\xi^*=\left(\begin{array}{cccc}0&e&-d&-c\\0&0&-b&-a\\0&0&0&0\\0&0&0&0\end{array}\right).
\]
One can check in this case that $\xi\xi^*=0$, but this contradicts the fact that $\xi\xi^*=\xi^2$ and $\xi^2\neq0$. This completes the proof of the lemma.
\end{proof}
\begin{pr}\label{abelian}
Let $(B,\omega_B)$ be the abelian symplectic Lie algebra of dimension 4. If $(\xi,b_0)$ is admissible then there exists a basis $\mathbb{B}=\{e_1,e_2,e_3,e_4\}$ of $B$ such that $\mathbb{B}$, $\xi$ and $b_0$ have one of these forms:
\begin{enumerate}
	\item $\xi=0$, $b_0\in B$ and $\{e_1,e_2,e_3,e_4\}$ is symplectic.
	\item $\{e_1,e_4,e_2,e_3\}$ is symplectic, 
\[
\xi=\left(\begin{array}{cccc}0&0&a&b\\0&0&c&d\\0&0&0&0\\0&0&0&0\end{array}\right) \mbox{ and }b_0=\alpha e_1+\beta e_2,\mbox{ where }a,b,c,d,\alpha,\beta\in\mathbb{R}\mbox{ with }ad-bc\neq0.
\]
\item $\{e_1,e_4,e_2,e_3\}$ is symplectic,
\[
\xi=\left(\begin{array}{cccc}0&0&0&a\\0&0&0&0\\0&0&0&0\\0&0&0&0\end{array}\right) \mbox{ and }b_0=\alpha e_1+\beta e_2+\gamma e_3,\mbox{ where }a,\alpha,\beta,\gamma\in\mathbb{R}\mbox{ with }a\neq0.
\]
\item $\{e_1,e_2,e_3,e_4\}$ is symplectic, 
\[
\xi=\left(\begin{array}{cccc}0&0&0&a\\0&0&0&0\\0&0&0&0\\0&0&0&0\end{array}\right) \mbox{ and }b_0=\alpha e_1+\beta e_3,\mbox{ where }a,\alpha,\beta\in\mathbb{R}\mbox{ with }a\neq0.
\]		
\end{enumerate} 	
\end{pr}
\begin{proof}
According to Lemma \ref{Lemma1}, $\xi^2=0$ and $Im(\xi)$ is totally isotropic. Thus $Im(\xi)\subset\ker\xi$ and $\dim Im(\xi)\leq2$.\\
{\it Case 1:} $Im(\xi)=\mathbb{R}e_1$, then we can choose a basis $\{e_1,e_2,e_3,e_4\}$ of $(B,\omega_B)$ such that $\{e_1,e_2,e_3\}$ is a basis of $\ker\xi$. Thus $\xi$ has the form 
	\[
	\xi=\left(\begin{array}{cccc}0&0&0&a\\0&0&0&0\\0&0&0&0\\0&0&0&0\end{array}\right)
	\]
	\begin{itemize}
		\item If $Im(\xi)=\left(\ker \xi\right)^\bot$, then we can choose the basis $\{e_1,e_4,e_2,e_3\}$ to be symplectic. In this case, $\xi^*$ has the form 
	\[
	\xi^*=\left(\begin{array}{cccc}0&0&0&-a\\0&0&0&0\\0&0&0&0\\0&0&0&0\end{array}\right)\mbox{ with }a\neq0.
	\]	
	Since $b_0\in\ker(\xi^*-\xi)$ then $b_0=\alpha e_1+\beta e_2+\gamma e_3$ with $\alpha, \beta, \gamma\in\mathbb{R}$.
	\item If there exists $x\in\ker\xi$ such that $\omega_B(e_1,x)\neq0$, then we can choose the basis $\{e_1,e_2,e_3,e_4\}$ to be symplectic. In this case, $\xi^*$ has the form 
	\[
	\xi^*=\left(\begin{array}{cccc}0&0&0&0\\0&0&0&0\\0&-a&0&0\\0&0&0&0\end{array}\right)\mbox{ with }a\neq0.
	\]	
	Since $b_0\in\ker(\xi^*-\xi)$ then $b_0=\alpha e_1+\beta e_3$ with $\alpha, \beta\in\mathbb{R}$. 
	\end{itemize}
{\it Case 2:} $Im(\xi)=\mathbb{R}e_1+\mathbb{R}e_2$, then $Im(\xi)=\ker\xi$ and there exists a symplectic basis $\{e_1,e_4,e_2,e_3\}$ such that 
\[
\xi=\left(\begin{array}{cccc}0&0&a&b\\0&0&c&d\\0&0&0&0\\0&0&0&0\end{array}\right) \mbox{ with }ad-bc\neq0.
\]
In this case, $\xi^*$ has the form 
\[
\xi^*=\left(\begin{array}{cccc}0&0&-d&-b\\0&0&-c&-a\\0&0&0&0\\0&0&0&0\end{array}\right)\mbox{ and }b_0=\alpha e_1+\beta e_2, \mbox{ with }ad-bc\neq0\mbox{ and }\alpha,\beta\in\mathbb{R}.
\]
\end{proof}

{\bf Second case: $B$ is not abelian.\\}
\begin{pr}\label{nonabelian}
Let $(B,\omega_B)$ be a non abelian flat symplectic Lie algebra of dimension 4. Then $(\xi,b_0)$ is admissible if and only if there exists a symplectic basis $\{e_1,e_4,e_2,e_3\}$ of $(B,\omega_B)$ such that 
\[
\xi=\left(\begin{array}{cccc}0&0&0&0\\0&0&0&0\\a&b&0&0\\c&d&0&0\end{array}\right)\mbox{ and }b_0=\alpha e_3+\beta e_4, \mbox{ with }a,b,c,d,\alpha,\beta\in\mathbb{R}, 
\]
or
\[
\xi=\left(\begin{array}{cccc}0&0&0&0\\0&0&0&0\\a&b&0&c\\0&d&0&0\end{array}\right)\mbox{ and }b_0=-9cde_1+xe_3+9d(a+d)e_4, \mbox{ with }a,b,d,x\in\mathbb{R}\mbox{ and }c\in\mathbb{R}^*. 
\]
\end{pr}
\begin{proof}
Let $(B,\omega_B)$ be a non abelian flat symplectic Lie algebra of dimension 4. Then $B$ is isomorphic to $\mathbb{R}\times\mathfrak{h}_3$: $[e_1,e_2]=e_3$ and $\{e_1,e_4,e_2,e_3\}$ is a symplectic basis. According to \eqref{products}, one has
\[
e_1e_2=\frac{2}{3}e_3,\ e_2e_1=-\frac{1}{3}e_3,\ e_2e_2=-\frac{1}{3}e_4.
\] 	
Recall that $(\xi,b_0)$ is admissible if and only if \eqref{eq1}-\eqref{eq5} hold. We put $\xi=(a_{ij})_{1\leq i,j\leq4}$. From $\xi\left([e_1,e_2]\right)=e_1\xi(e_2)-e_2\xi(e_1)$ we deduce that $a_{13}=a_{23}=0$, $3a_{33}=a_{11}+2a_{22}$ and $3a_{43}=a_{21}$. In the same way, by replacing $[a,b]$ with $[e_2,e_4]$ in \eqref{eq1} we get $a_{14}=a_{24}=0$. Thus  
\[
\xi=\left(\begin{array}{cccc}a_{11}&a_{12}&0&0\\a_{21}&a_{22}&0&0\\a_{31}&a_{32}&\frac{a_{11}+2a_{22}}{3}&a_{34}\\a_{41}&a_{42}&\frac{a_{21}}{3}&a_{44}\end{array}\right)
\]
Since $\{e_1,e_4,e_2,e_3\}$ is a symplectic basis then
\[
\xi^*=\left(\begin{array}{cccc}a_{44}&a_{34}&0&0\\\frac{a_{21}}{3}&\frac{a_{11}+2a_{22}}{3}&0&0\\-a_{42}&-a_{32}&a_{22}&a_{12}\\-a_{41}&-a_{31}&a_{21}&a_{11}\end{array}\right)
\]
In \eqref{eq2}, if we take $a=b=e_1$, we get $a_{21}=0$. For $a=e_1$ and $b=e_2$, we get 
\begin{eqnarray}\label{E1}
3a_{44}=4a_{11}+2a_{22}.
\end{eqnarray}
 For $a=e_2$ and $b=e_1$, we get 
 \begin{eqnarray}\label{E2}
 3a_{44}=a_{11}+5a_{22}.
 \end{eqnarray}
 Finaly, for $a=e_2$ and $b=e_2$, we get $a_{12}=0$. From \eqref{E1}-\eqref{E2} we deduce that $a_{22}=a_{11}$ and $a_{44}=2a_{11}$. Let $b_0=\alpha e_1+\beta e_2+\gamma e_3+\delta e_4$. We have
\[
\Rr_{b_0}=\left(\begin{array}{cccc}0&0&0&0\\0&0&0&0\\\frac{2\beta}{3}&-\frac{\alpha}{3}&0&0\\0&-\frac{\beta}{3}&0&0\end{array}\right)
\]
Accoding to \eqref{eq3}, we deduce that $a_{11}=0$, $\alpha=-9a_{34}a_{42}$ and $\beta=9a_{34}a_{41}$. From \eqref{eq4}, we deduce that $\beta=0$.
\begin{itemize}
	\item If $a_{34}=0$, then $\alpha=0$ and hence $(\xi,b_0)$ is admissible in this case and has the form
\[
\xi=\left(\begin{array}{cccc}0&0&0&0\\0&0&0&0\\a&b&0&0\\c&d&0&0\end{array}\right)\mbox{ and }b_0=x e_3+y e_4, \mbox{ with }a,b,c,d,x,y\in\mathbb{R}. 
\]
	\item If $a_{34}\neq0$ then $a_{41}=0$ and $\alpha=-9a_{34}a_{42}$. The equation \eqref{eq2} is equivalent to $\delta=9a_{42}(a_{31}+a_{42})$. In this case, $(\xi,b_0)$ is admissible and has the form
	\[
	\xi=\left(\begin{array}{cccc}0&0&0&0\\0&0&0&0\\a&b&0&c\\0&d&0&0\end{array}\right)\mbox{ and }b_0=-9cde_1+xe_3+9d(a+d)e_4, \mbox{ with }a,b,d,x\in\mathbb{R}\mbox{ and }c\in\mathbb{R}^*. 
	\]
\end{itemize}
\end{proof}
Using Proposition \ref{abelian} and Proposition \ref{nonabelian}, we can determine all flat symplectic Lie algebras of dimension 6.
\begin{theo}\label{dim6}
	A 6-dimensional Lie algebra $\G$ admits a flat symplectic form if and only if $\G$ is isomorphic to one and only one of the following Lie algebras:
	\begin{itemize}
		\item $\mathbb{R}^6$: The 6-dimensional abelian Lie algebra.
		\item $\mathbb{R}^3\times\mathfrak{h}_3$: $[x_1,x_2]=x_6$.
		\item $\G_6^1$: $[x_1,x_2]=x_4$, $[x_1,x_3]=x_5$, $[x_2,x_3]=x_6$.
		\item $\G_6^2$: $[x_1,x_2]=x_5$, $[x_1,x_3]=x_6$.
		\item $\G_6^3$: $[x_1,x_2]=x_4$, $[x_1,x_3]=x_5$, $[x_1,x_4]=x_6$, $[x_2,x_3]=x_6$.
	\end{itemize}
\end{theo}
\begin{proof}
	Any 6-dimensional flat symplectic Lie algebra $(\G,\omega)$ is a double extension of a 4-dimensional flat symplectic Lie algebra $(B,\omega_B)$ by means of an admissible couple $(\xi,b_0)$.\\
	{\it Case 1:} If $B$ is abelian, then $(\xi,b_0)$ has one of the forms given in Proposition \ref{abelian}. 
	\begin{itemize}
		\item If $\xi=0$ and $b_0=\alpha e_1+\beta e_2+\gamma e_3+\delta e_4\in B$, then according to \eqref{Liebrackets}, one has
		\[
		[\bar{e},e_1]=-\delta e,\ [\bar{e},e_2]=-\gamma e,\ [\bar{e},e_3]=\beta e,\ [\bar{e},e_4]=\alpha e.
		\]
		If $\alpha=\beta=\gamma=\delta=0$, then $\G$ is the abelian Lie algebra $\mathbb{R}^6$. If, for example, $\alpha\neq0$, then we put $x_1=\bar{e}$, $x_2=e_4$, $x_6=\alpha e$, $x_3=\alpha e_3-\beta e_4$, $x_4=\alpha e_2+\gamma e_4$ and $x_5=\alpha e_1+\delta e_4$. Thus, the only non vanishing Lie bracket in this basis is $[x_1,x_2]=x_6$. It follows that $\G$ is isomorphic in this case to $\mathbb{R}^3\times\mathfrak{h}_3$. 
		\item If 
		\[
		\xi=\left(\begin{array}{cccc}0&0&a&b\\0&0&c&d\\0&0&0&0\\0&0&0&0\end{array}\right)\mbox{ and }b_0=\alpha e_1+\beta e_2, 
		\]
		where $\{e_1,e_4,e_2,e_3\}$ is symplectic and $ad-bc\neq0$. Then 
\[
\xi^*=\left(\begin{array}{cccc}0&0&-d&-b\\0&0&-c&-a\\0&0&0&0\\0&0&0&0\end{array}\right), 
\]
and according to \eqref{Liebrackets}, one has
\[
[\bar{e},e_3]=-(2a+d)e_1-3ce_2+\beta e,\ [\bar{e},e_4]=-3be_1-(a+2d)e_2+\alpha e,\ [e_3,,e_4]=(a-d)e.
\] 	
If the vectors $e_1'=-(2a+d)e_1-3ce_2+\beta e$, $e_2'=-3be_1-(a+2d)e_2+\alpha e$ and $e'=(a-d)e$ are linearly independent then $[\bar{e},e_3]=e_1'$, $[\bar{e},e_4]=e_2'$ and $[e_3,e_4]=e'$ and hence $\G$ is isomorphic in this case to $\G_6^1$. If $rank\{e_1',e_2',e'\}=2$, for example, $\{e_1',e_2'\}$ are linearly independent and $e'=\lambda_1e_1'+\lambda_2e_2'$, then $[\bar{e},e_3]=e_1'$, $[\bar{e},e_4]=e_2'$ and $[e_3,e_4]=\lambda_1e_1'+\lambda_2e_2'$. We put $e_3'=e_3-\lambda_2\bar{e}$ and $e_4'=e_4+\lambda_1\bar{e}$. In this basis, the only non vanishing Lie brackets are $[\bar{e},e_3']=e_1'$, $[\bar{e},e_4']=e_2'$, and hence $\G$ is isomorphic in this case to $\G_6^2$. If $rank\{e_1',e_2',e'\}=1$, then it is easy to check that $\G$ is isomorphic to $\mathbb{R}^3\times\mathfrak{h}_3$.
\item If 
\[
\xi=\left(\begin{array}{cccc}0&0&0&a\\0&0&0&0\\0&0&0&0\\0&0&0&0\end{array}\right)\mbox{ and }b_0=\alpha e_1+\beta e_2+\gamma e_3, 
\]
where $\{e_1,e_4,e_2,e_3\}$ is symplectic and $a\neq0$. Then $\xi^*=-\xi$ and $\G$ is isomorphic in this case to $\mathbb{R}^3\times\mathfrak{h}_3$ if $\beta=\gamma=0$, or to $\G_6^2$ if $(\beta,\gamma)\neq(0,0)$.
\item If 
\[
\xi=\left(\begin{array}{cccc}0&0&0&a\\0&0&0&0\\0&0&0&0\\0&0&0&0\end{array}\right)\mbox{ and }b_0=\alpha e_1+\beta e_3, 
\]
where $\{e_1,e_2,e_3,e_4\}$ is symplectic and $a\neq0$. Then
\[
\xi^*=\left(\begin{array}{cccc}0&0&0&0\\0&0&0&0\\0&-a&0&0\\0&0&0&0\end{array}\right)
\]
and the Lie backets are given by 
\[
[\bar{e},e_2]=-ae_3-\beta e,\ [\bar{e},e_4]=-2ae_1+\alpha e,\ [e_2,e_4]=-ae. 
\]
Since $a\neq0$, then by putting $e_3'=-ae_3-\beta e$, $e_4'=-2ae_1+\alpha e$ and $e'=-ae$ one can check that $\G$ is isomorphic in this case to $\G_6^1$.
\end{itemize}
 {\it Case 2:} If $B$ is not abelian, then $(\xi,b_0)$ has one of the two forms given in Proposition \ref{nonabelian}. 
 \begin{itemize}
 	\item If 
\[
\xi=\left(\begin{array}{cccc}0&0&0&0\\0&0&0&0\\a&b&0&0\\c&d&0&0\end{array}\right)\mbox{ and }b_0=\alpha e_3+\beta e_4, 
\]
where $\{e_1,e_4,e_2,e_3\}$ is symplectic. Then
\[
\xi^*=\left(\begin{array}{cccc}0&0&0&0\\0&0&0&0\\-d&-b&0&0\\-c&-a&0&0\end{array}\right)
\]
and the Lie backets are given by 
\[
[\bar{e},e_1]=-(2a+d)e_3-3ce_4-\beta e,\ [\bar{e},e_2]=-3be_3-(a+2d)e_4-\alpha e,\ [e_1,e_2]=e_3+(d-a)e. 
\]
As we showed in the abelian case, one can check that $\G$ is isomorphic to $\mathbb{R}^3\times\mathfrak{h}_3$, $\G_6^1$ or $\G_6^2$.
\item If 
\[
\xi=\left(\begin{array}{cccc}0&0&0&0\\0&0&0&0\\a&b&0&c\\0&d&0&0\end{array}\right)\mbox{ and }b_0=-9cde_1+xe_3+9d(a+d)e_4, 
\]
where $\{e_1,e_4,e_2,e_3\}$ is symplectic and $c\neq0$. Then
 	\[
 	\xi^*=\left(\begin{array}{cccc}0&c&0&0\\0&0&0&0\\-d&-b&0&0\\0&-a&0&0\end{array}\right)
 	\]
 	and the Lie brackets are given by 
 	\begin{eqnarray*}
 	[\bar{e},e_1]&=&-(2a+d)e_3-9d(a+d)e,\ [\bar{e},e_2]=ce_1-3be_3-(a+2d)e_4-xe,\\
 \ 	[\bar{e},e_4]&=&-2ce_3-9cde,\ [e_1,e_2]=e_3+(d-a)e,\ [e_2,e_4]=ce.
 	\end{eqnarray*}
 By taking, $\bar{f}=\frac{\bar{e}}{c}$, $f_1=e_1-\frac{3b}{c}e_3-\frac{a+2d}{c}e_4-\frac{x}{c}e$, $f_2=e_2$, $f_3=e_3$, $f_4=-\frac{e_4}{2}$ and $f=-\frac{c}{2}e$. We reduce the Lie brackets to
 \begin{eqnarray*}
 	[\bar{f},f_1]&=&\frac{3d}{c}f_3-\frac{18d^2}{c^2}f,\ [\bar{f},f_2]=f_1,\\
 	\ 	[\bar{f},f_4]&=&f_3-\frac{9d}{c}f,\ [f_1,f_2]=f_3-\frac{6d}{c}f,\ [f_2,f_4]=f.
 \end{eqnarray*}
We replace $f_3$ by $f_3'=f_3-\frac{6d}{c}f$ and $\bar{f}$ by $\bar{f}'=\bar{f}+\frac{3d}{c}f_2$. Thus, the Lie brackets will be 
\[
[\bar{f}',f_2]=f_1,\ [\bar{f}',f_4]=f_3',\ [f_1,f_2]=f_3',\ [f_2,f_4]=f.
\]
By putting $x_1=f_2$, $x_4=-f_1$, $x_6=f_3'$, $x_5=f$, $x_3=f_4$, $x_2=\bar{f}'$, one can see that $\G$ is isomorphic to $\G_6^3$. This completes the proof of the theorem.
 \end{itemize}
\end{proof}
\begin{rem}
\begin{itemize}
	\item In Theorem \ref{dim6}, the Lie algebras $\mathbb{R}^3\times\mathfrak{h}_3$, $\G_6^1$ or $\G_6^2$ are 2-step nilpotent and $\G_6^3$ is 3-step nilpotent.
	\item We can also determine the flat symplectic form on the Lie algebras given in Theorem \ref{dim6} by making the same changes we made for the basis $\{e,e_1,e_2,e_3,e_4,\bar{e}\}$ to the initial flat symplectic form $e^*\wedge\bar{e}^*+\omega_B$.
		\item We have preferred to describe the Lie algebras in Theorem \ref{dim6} by giving the same Lie brackets in the classification of symplectic nilpotent Lie algebras given in page 51 of \cite{khakim}. Note that the Lie algebras $\mathbb{R}^3\times\mathfrak{h}_3$, $\G_6^1$, $\G_6^2$ and $\G_6^3$ are respectively denoted in \cite{khakim} by classes 25, 18, 23 and 13.   
\end{itemize}	
\end{rem}
We give here examples of flat symplectic forms on the Lie algebras mentioned in Theorem \ref{dim6} and the coresponding product $\bullet$. We use the same notations in \cite{khakim}.  
\vspace{0.7cm}

\begin{small}
\begin{tabular}{|c|c|c|c|}
\hline	
{\scriptsize{\bf Lie algebras}}&{\scriptsize{\bf Non-vanishing Lie brackets}}&{\scriptsize{\bf Flat symplectic forms}}&{\scriptsize{\bf Non-vanishing products $\bullet$}}\\\hline$\mathbb{R}^6$&&$x_1^*\wedge x_6^*+x_2^*\wedge x_5^*+x_3^*\wedge x_4^*$&\\\hline$\mathbb{R}^3\times\mathfrak{h}_3$&$[x_1,x_2]=x_6$&$x_1^*\wedge x_6^*+x_2^*\wedge x_5^*+x_3^*\wedge x_4^*$&$x_1\bullet x_1=\frac{1}{3}x_5,\ x_1\bullet x_2=\frac{1}{3}x_6$,\ \\&&&$x_2\bullet x_1=-\frac{2}{3}x_6$.\\\hline
&&&$x_1\bullet x_2=\frac{1-2\lambda}{3-3\lambda}x_4,\ x_1\bullet x_3=\frac{2\lambda-1}{3\lambda}x_5$\\$\G_6^1$&$[x_1,x_2]=x_4$, $[x_1,x_3]=x_5$,&$x_1^*\wedge x_6^*+\lambda x_2^*\wedge x_5^*+(\lambda-1)$&$x_2\bullet x_1=\frac{\lambda-2}{3-3\lambda}x_4,\ x_2\bullet x_3=\frac{2-\lambda}{3}x_6$,\\&$[x_2,x_3]=x_6$&$x_3^*\wedge x_4^*$,  $\lambda\in\mathbb{R}-\{0,1\}.$&$x_3\bullet x_1=\frac{-\lambda-1}{3\lambda}x_5,\ x_3\bullet x_2=\frac{-1-\lambda}{3}x_6$,\\\hline&&&$x_1\bullet x_1=\frac{1}{3}x_4,\ x_1\bullet x_2=\frac{2}{3}x_5$,\\$\G_6^2$&$[x_1,x_2]=x_5$, $[x_1,x_3]=x_6$&$x_1^*\wedge x_6^*+x_2^*\wedge x_5^*+x_3^*\wedge x_4^*$&$x_1\bullet x_3=\frac{1}{3}x_6,\ x_2\bullet x_1=-\frac{1}{3}x_5$,\\&&&$x_2\bullet x_2=-\frac{1}{3}x_6,\ x_3\bullet x_1=-\frac{2}{3}x_6$.\\\hline&&&$x_1\bullet x_1=\frac{2}{3}x_3,\ x_1\bullet x_4=\frac{1}{3}x_6$,\\$\G_6^3$&$[x_1,x_2]=x_4$, $[x_1,x_3]=x_5$&$x_1^*\wedge x_6^*+\frac{1}{2}x_2^*\wedge x_5^*$&$x_2\bullet x_1=-x_4,\ x_2\bullet x_3=\frac{1}{2}x_6$,\\&$[x_1,x_4]=x_6$, $[x_2,x_3]=x_6$&$-\frac{1}{2}x_3^*\wedge x_4^*$&$x_3\bullet x_1=-x_5,\ x_3\bullet x_2=\frac{1}{6}x_6$,\\&&&$x_4\bullet x_1=-\frac{2}{3}x_6$.\\\hline
\end{tabular}
\end{small}

\bibliographystyle{elsarticle-num}

\end{document}